\documentclass[a4paper]{amsart}
\usepackage[T1]{fontenc}
\usepackage[latin1]{inputenc}
\usepackage[english]{babel}
\usepackage{amsmath,amsthm,latexsym,amssymb,mathrsfs}
\DeclareMathAlphabet{\altmathcal}{OMS}{cmsy}{m}{n}
\usepackage{mathptmx}
\usepackage{dsfont} 
\usepackage{mathtools}
\usepackage{color}
\usepackage{physics}
\usepackage{faktor}
\usepackage[shortlabels]{enumitem}
\usepackage{hyperref}
\usepackage{graphicx}
\usepackage{tabularx}
\newcolumntype{C}{>{\centering\arraybackslash}X} 
\usepackage{csquotes}
\usepackage{url}
\usepackage{import}
\usepackage[pass]{geometry}

\usepackage[style=alphabetic,natbib=true,url=false,doi=false,isbn=false,backend=biber]{biblatex}
\renewbibmacro{in:}{}
\usepackage{accents}

\newtheorem*{theorem*}{Theorem}

\newtheorem{theorem}{Theorem}[section]

\newtheorem{theoremx}{Theorem}

\newtheorem{lemma}[theorem]{Lemma}
\newtheorem{corollary}[theorem]{Corollary}
\newtheorem{proposition}[theorem]{Proposition}
\newtheorem*{proposition*}{Proposition}
\theoremstyle{definition}
\newtheorem{definition}[theorem]{Definition}

\theoremstyle{remark}
\newtheorem{remark}[theorem]{Remark}

\newcommand{\defin}{\vcentcolon =}
\newcommand{\C}{\mathbb{C}}
\newcommand{\R}{\mathbb{R}}
\newcommand{\N}{\mathbb{N}}

\newcommand{\I}{I}
\newcommand{\II}{I \!\! I}

\newcommand{\Hyp}{\mathbb{H}}

\newcommand{\Teich}{\altmathcal{T}}

\newcommand{\length}{\ell}
\newcommand{\CC}{C}
\newcommand{\neigh}{N}
\newcommand{\Dlie}{\altmathcal{L}}

\newcommand{\WP}{\textit{WP}}

\newcommand{\id}{\textit{id}}
\newcommand{\1}{\mathds{1}}

\newcommand{\MesLam}{\altmathcal{ML}}

\newcommand{\QD}{\altmathcal{QD}}

\newcommand{\mappa}[3]{#1 \colon #2 \rightarrow #3}
\newcommand{\hsk}{\hskip0pt}

\DeclarePairedDelimiterX{\scal}[2]{\langle}{\rangle}{#1, #2}
\DeclarePairedDelimiterX{\scall}[2]{(}{)}{#1, #2}

\DeclareMathOperator{\arctanh}{arctanh}

\DeclareMathOperator{\Area}{Area}

\DeclareMathOperator{\ext}{ext}
\DeclareMathOperator{\grd}{grad}

\DeclareMathOperator{\Ric}{Ric}

\DeclareMathOperator{\divr}{div}

\title[The infimum of the dual volume of convex co-compact \texorpdfstring{$3$}{3}-manifolds]{The infimum of the dual volume of convex co-compact hyperbolic \texorpdfstring{$3$}{3}-manifolds}

\author{Filippo Mazzoli}

\date{\today}

\address{Department of Mathematics \\
    University of Virginia \\
    141 Cabell Drive \\
    Charlottesville VA 22904 \\
    U. S. A.}



\email{filippomazzoli@me.com}

\begin{document}
    
\begin{abstract}
    \noindent We show that the infimum of the dual volume of the convex core of a convex co-compact hyperbolic $3$-\hsk manifold with incompressible boundary coincides with the infimum of the Riemannian volume of its convex core, as we vary the geometry by quasi-\hsk isometric deformations. We deduce a linear lower bound of the volume of the convex core of a quasi-\hsk Fuchsian manifold in terms of the length of its bending measured lamination, with optimal multiplicative constant.
\end{abstract}

\maketitle

\section*{Introduction}

Let $M$ be a complete hyperbolic $3$-\hsk manifold, and let $\CC M$ be its convex core, namely the smallest non-\hsk empty convex subset of $M$. Then $M$ is said to be \emph{convex co-\hsk compact} if $\CC M$ is a compact subset. The notion of dual volume of the convex core $V_\CC^*(M)$ arises from the polarity correspondence between the hyperbolic and the de Sitter spaces (see \cite[Section 1]{schlenker2002hypersurfaces}, \cite{mazzoli2020thesis}). If $M$ is a convex co-\hsk compact hyperbolic $3$-\hsk manifold, then $V_\CC^*(M)$ coincides with $V_\CC(M) - \frac{1}{2} \length_m(\mu)$, where $V_\CC(M)$ stands for the usual Riemannian volume of the convex core, and $\length_m(\mu)$ denotes the length of the bending measured lamination $\mu$ with respect to the hyperbolic metric $m$ of the boundary of the convex core of $M$. 
The aim of this paper is to study the infimum of $V_\CC^*$, considered as a function over the space $\QD(M)$ of quasi-\hsk isometric deformations of a given convex co-\hsk compact hyperbolic $3$-\hsk manifold $M$ with incompressible boundary. In particular, we will prove

\begin{theoremx} \label{THM:INF_DUAL_VOLUME}
	For every convex co-\hsk compact hyperbolic $3$-\hsk manifold $M$ with incompressible boundary we have
	\[
	\inf_{M' \in \QD(M)} V_\CC^*(M') = \inf_{M' \in \QD(M')} V_\CC(M') .
	\]
	Moreover, $V_\CC^*(M') = V_\CC(M')$ if and only if the boundary of the convex core of $M'$ is totally geodesic.
\end{theoremx}

When $M$ is a quasi-\hsk Fuchsian manifold, Theorem \ref{THM:INF_DUAL_VOLUME} can be equivalently stated as
\begin{equation} \label{eq:lower_bound_volume}
V_\CC(M') \geq \frac{1}{2} \length_{m'}(\mu') 
\end{equation}
for every $M' \in \QD(M)$, where $\length_{m'}(\mu')$ is the length of the bending measure of $\partial \CC M'$. As a consequence of the variation formulae of $V_\CC$ \cite{bonahon1998schlafli} and of $V_\CC^*$ \cite{mazzoli2018the_dual} (see also \cite{krasnov2009symplectic}), we will see in Corollary \ref{cor:optimal_constant} that the multiplicative constant $1/2$ appearing here is optimal, and it is realized near the Fuchsian locus. 

Theorem \ref{THM:INF_DUAL_VOLUME} is to the dual volume as the following result of Bridgeman, Brock, and Bromberg is to the renormalized volume:

\begin{theorem*}[{\cite[Theorem 3.10]{bridgeman_brock_bromberg2017}}]
    For every convex co-\hsk compact hyperbolic $3$-\hsk manifold $M$ with incompressible boundary we have
    \[
    \inf_{M' \in \QD(M)} V_R(M') = \inf_{M' \in \QD(M)} V_\CC(M') .
    \]
    Moreover, $V_R(M') = V_\CC(M')$ if and only if the boundary of the convex core of $M$ is totally geodesic.
\end{theorem*}

By the work of Thurston, if the compact $3$-manfiold with boundary $N : = M \cup \partial_\infty M$ is acylindrical, then there exists a unique convex co-compact structure $M_0 \in \QD(M)$ whose convex core has totally geodesic boundary. In \cite{storm2007hyperbolic} (see also \cite{storm2002minimal}) Storm proved that the infimum of the volume of the convex core function $V_C : \QD(M) \to \R$ is equal to half the simplicial volume of the doubled manifold $D(N)$. Moreover, the infimum is realized exactly when $N$ is acylindrical, and it is achieved at $M_0$. Theorem \ref{THM:INF_DUAL_VOLUME} and \cite[Theorem 3.10]{bridgeman_brock_bromberg2017} then imply that the same characterization holds true for the infimum of the dual volume and the renormalized volume, respectively. In the case of the renormalized volume $V_R$, such description of $\inf V_R$ was first established by Pallete \cite{pallete2016continuity}, without making use of Storm's result. Bridgeman, Brock, and Bromberg \cite{bridgeman2021weilpetersson} recently introduced a notion of surgered gradient flow of the renormalized volume in the relatively acylindrical case, which allowed them to obtain new comparisons between the renormalized volume and the Weil-Petersson geometry of the deformation spaces of convex cocompact $3$-manifolds, generalizing in particular the works of Brock \cite{brock2003the_weil} and Schlenker \cite{schlenker2013renormalized}. In the same work, a new proof of Storm's result in the acylindrical case is obtained by the authors as a biproduct of their analysis (see in particular \cite[Corollary 6.5]{bridgeman2021weilpetersson}).

Dual volume, renormalized volume and Riemannian volume of the convex core are related by the following chain of inequalities:
\[
V_\CC^*(M) \defin V_\CC(M) - \frac{1}{2} \length_m(\mu) \leq V_R(M) \leq V_\CC(M) - \frac{1}{4} \length_m(\mu) \leq V_\CC(M) .
\]
Here the second inequality is due to Schlenker \cite{schlenker2013renormalized}, and the lower bound of $V_R$ is proved in \cite[Theorem 3.7]{bridgeman_brock_bromberg2017}. Observe in particular that Theorem \ref{THM:INF_DUAL_VOLUME} implies the aforementioned result \cite[Theorem 3.10]{bridgeman_brock_bromberg2017} concerning the infimum of the renormalized volume. The request on $M$ to have incompressible boundary is necessary, indeed it has been shown by Pallete \cite{pallete2019upper} that there exist Schottky groups with negative renormalized volume.

\vspace{1em}

The proof of Theorem \ref{THM:INF_DUAL_VOLUME} we present here broadly follows the same strategy of the work of Bridgeman, Brock, and Bromberg \cite{bridgeman_brock_bromberg2017}, with some necessary differences: the authors of \cite{bridgeman_brock_bromberg2017} interpret the renormalized volume as a function $V_R$ over the Teichm\"uller space $\Teich(\partial_\infty M)$ of the \emph{conformal boundary at infinity} of $M$ (by the works of Bers, Maskit, and Kra \cite{bers1970spaces,maskit1971self-maps,kra1972on_spaces}), and they estimate the difference $\abs{V_R - V_\CC}$ as one follows the (opposite of the) Weil-\hsk Petersson gradient flow of $V_R$ on $\Teich(\partial_\infty M)$. In order to study the dual volume function, the analogy between the variation formula of the renormalized volume (see the work of Krasnov and Schlenker \cite[Lemma 5.8]{krasnov2008renormalized}, or Section \ref{subsec:dual_volume}) and the dual Bonahon-\hsk Schl\"afli formula \cite{mazzoli2018the_dual} would tempt us to consider $V_\CC^*$ as a function of the Teichm\"uller space $\Teich(\partial \CC M)$, seen as deformation space of \emph{hyperbolic structures} on the boundary of the convex core of $M$. However, the hyperbolic structure on $\partial \CC M$ is only conjecturally thought to provide a parametrization of the quasi-\hsk isometric deformation space of $M$. To avoid this difficulty, we rather focus our attention of a family of functions $V_k^*$ approximating $V_\CC^*$, for which a similar procedure is possible. 

Given $k$ a real number in the interval $(-1,0)$, we say that an embedded surface $\Sigma_k \subset M$ is a \emph{$k$-\hsk surface} if its first fundamental form (namely the restriction of the metric of $M$ on the tangent space to $\Sigma_k$) is a Riemannian metric with constant Gaussian curvature equal to $k$. Then, by the work of Labourie \cite{labourie1991probleme}, the complementary region of the convex core of $M$ is foliated by $k$-\hsk surfaces, which converge to $\partial \CC M$ as $k$ goes to $-1$, and tend towards the conformal boundary at infinity $\partial_\infty M$ as $k$ goes to $0$. The function $V_k^*(M)$ is then defined to be the dual volume of the region $M_k$ of $M$ enclosed by its $k$-\hsk surfaces, one per each geometrically finite end of $M$. By the works of Labourie \cite{labourie1992metriques} and Schlenker \cite{schlenker2006hyperbolic}, the hyperbolic structures of the $k$-\hsk surfaces do provide a parametrization of $\QD(M)$, fact that allows us to study $V_k^*$ as a function over the Teichm\"uller space of $\partial M_k$. At this point, studying the Weil-\hsk Petersson gradient of $V_k^*$ on $\Teich(\partial M_k)$, we prove that the difference between the dual volume and the standard volume of the regions $M_k$ is well-\hsk behaved as one follows backwards the lines of the flow, and finally we deduce the statement of Theorem \ref{THM:INF_DUAL_VOLUME} by taking a limit for $k$ that goes to $-1$. While the methods of \cite{bridgeman_brock_bromberg2017} for the study of the renormalized volume heavily rely on the relations between the geometry of the boundary of the convex core and the properties of the \emph{Schwarzian at infinity} of $\partial_\infty M$, here we use a more analytical approach to determine the necessary bounds on the geometric quantities of the $k$-\hsk surfaces $\partial_k M$ of $M$, which will guarantee us the existence and the good behavior of the flow of the Weil-\hsk Petersson gradient vector fields of $V_k^*$.

\subsection*{Outline of the paper} After the first section of background, we suggest the reader to initially move backwards (as for the flow of the gradient of the functions $V_k^*$) while going through this exposition: in Section \ref{sec:proof_thmA} the proof of Theorem \ref{THM:INF_DUAL_VOLUME} is described. Here the analogy with the work of Bridgeman, Brock, and Bromberg \cite{bridgeman_brock_bromberg2017} is manifest, the required technical ingredients (Lemma \ref{lem:lower_bound_gradient}, Corollary \ref{cor:flow_always_def} and Lemma \ref{lem:lower_bound_Vk*}) are formally very similar to the ones developed for the renormalized volume. 

Section \ref{sec:gradient_dual_volume} focuses on the study of the Weil-\hsk Petersson gradient $\grd_\WP V_k^*$ of the dual volume functions $V_k^*$ and the proofs of the ingredients mentioned above: in Lemma \ref{lem:lower_bound_gradient} we determine a lower bound of the norm of $\grd_\WP V_k^*$ in terms of the integral of the mean curvature of $\partial M_k$ (which replaces the role of the length $\length_m(\mu)$ in the definition of the dual volume of the regions $M_k$). In Corollary \ref{cor:flow_always_def} we show that the flow of the vector field $\grd_\WP V_k^*$ is defined for all times, and in Lemma \ref{lem:lower_bound_Vk*} we prove the existence of a global lower bound of the dual volumes $V_k^*$ over $\QD(M)$. All the proofs of this section rely on differential-\hsk geometric methods and are consequences of an explicit description of the Weil-\hsk Petersson gradient of $V_k^*$ developed in Proposition \ref{prop:gradient_dual_volume}. This presentation of the vector field $\grd_\WP V_k^*$ is inspired by an orthogonal decomposition of the space of symmetric tensors due to Fischer and Marsden \cite{fischer1975deformations}, and it involves the solution $u_k$ of a simple PDE (equation \eqref{eq:pde}) over the $k$-\hsk surface $\partial M_k$. In particular, the proof of Corollary \ref{cor:flow_always_def} will require us to have a uniform control of the $\mathscr{C}^2$-\hsk norm of the function $u_k$. Section \ref{sec:some_useful_estimates} (and in particular Lemma \ref{lem:estimates_uk}) provides us this last ingredient, and it is essentially based on the classical regularity theory for linear elliptic differential operators (see e. g. \cite{evans1998partial}), and on the following property of $k$-surfaces:
    \begin{proposition*}[{see Proposition \ref{prop:bounds_mean_curvature}}]
        For any $k \in (0,1)$ and $n \in \N$, there exists a positive constant $N_{k,n}$ such that for every convex co-compact hyperbolic $3$-manifold $M$ and for every incompressible $k$-surface $\Sigma_k$ in $M$, the $\mathscr{C}^n$-norm of the mean curvature of $\Sigma_k$ is bounded above by $N_{n,k}$.
    \end{proposition*}
    The existence of such universal upper bound was proved (with weaker assumptions than the ones appearing above) by Bonsante, Danciger, Maloni, and Schlenker in \cite[Proposition 3.8]{bonsante2019induced} for $n = 0$ (and the same strategy actually shows that the statement holds for any $n$), and its proof heavily relies on a compactness criterion for isometric immersions of surfaces established by Labourie \cite{labourie1991probleme} (see also Bonsante, Danciger, Maloni, and Schlenker \cite[Proposition 3.6]{bonsante2019induced}). As it will be manifest in the proof of Proposition \ref{prop:bounds_mean_curvature}, the constants $N_{n,k}$ that we will produce are unfortunately not explicit.

\subsection*{Acknowledgments} This work is extracted from my Ph.D. thesis \cite{mazzoli2020thesis}. I would like to thank my advisor Jean-\hsk Marc Schlenker for his help during my doctoral studies in Luxembourg. I am grateful also to Martin Bridgeman, Jeffrey Brock, and Kenneth Bromberg, together with the GEAR Network, for giving me the opportunity to visit them and for our useful conversations, and to Sara Maloni, for bringing to my attention the discussion on the optimality of the multiplicative constant in \eqref{eq:lower_bound_volume}. Finally, I would like to thank the anonymous referees for their useful remarks and suggestions. This work has been partially supported by the Luxembourg National Research Fund PRIDE15/10949314/GSM/Wiese and by the U.S. National Science Foundation grants DMS 1107452, 1107263, 1107367 "RNMS: Geometric Structures and Representation Varieties" (the GEAR Network).

\section{Preliminaries} 

\subsection{Hyperbolic \texorpdfstring{$3$}{3}-manifolds}

Let $M$ be an orientable complete hyperbolic $3$-\hsk manifold, namely a complete Riemannian $3$-\hsk manifold with constant sectional curvature equal to $-1$, and let $\Gamma$ be a discrete and torsion-\hsk free group of orientation-\hsk preserving isometries of the hyperbolic $3$-\hsk space $\Hyp^3$, such that $M$ is isometric to $\Hyp^3 / \Gamma$. We define the \emph{limit set} of $\Gamma$ to be
\[
\Lambda_\Gamma \defin \overline{\Gamma \cdot x_0} \cap \partial_\infty \Hyp^3 ,
\]
where $\overline{\Gamma \cdot x_0}$ denotes the closure of the $\Gamma$-\hsk orbit of $x_0$ in $\overline{\Hyp}{}^3 \defin \Hyp^3 \cup \partial_\infty \Hyp^3$. It is simple to see that the definition of $\Lambda_\Gamma$ does not depend on the choice of the basepoint $x_0 \in \Hyp^3$. If $\Gamma$ is non-\hsk elementary (i. e. it does not have any finite orbit in $\overline{\Hyp}{}^3$), then $\Lambda_\Gamma$ can be characterized as the smallest closed $\Gamma$-\hsk invariant subset of $\partial_\infty \Hyp^3$ (see e. g. \cite[Chapter 12]{ratcliffe2006hyperbolic}). The complementary region $\Omega_\Gamma$ of the limit set in $\partial_\infty \Hyp^3$ is called the \emph{domain of discontinuity} of $\Gamma$.

\subsection{The convex core}

If $\mappa{\pi}{\Hyp^3}{\Hyp^3 / \Gamma \cong M}$ denotes the universal cover of $M$, then a subset $C$ of $M$ is \emph{convex} if and only if $\pi^{-1}(C)$ is convex in $\Hyp^3$. If $\Gamma$ is non-\hsk elementary, then every non-\hsk empty $\Gamma$-\hsk invariant convex subset of $\Hyp^3$ contains the \emph{convex hull} $\CC_\Gamma$ of $\Gamma$, which consists of the intersection of all half-\hsk spaces $H$ of $\Hyp^3$ satisfying $\overline{H} \supseteq \Lambda_\Gamma$ ($\overline{H}$ stands for the closure of $H$ in $\overline{\Hyp}{}^3$). The image $\CC M \defin \pi(\CC_\Gamma)$ describes a convex subset of $M$, called the \emph{convex core} of $M$, which is minimal among the family of non-\hsk empty convex subsets of $M$.

Let now $M$ be a \emph{convex co-\hsk compact} hyperbolic $3$-\hsk manifold, namely a non-\hsk compact complete hyperbolic $3$-\hsk manifold whose convex core is compact. The boundary of the convex core $\partial \CC M$ of $M$ is the union of a finite collection of connected surfaces, each of which is totally geodesic outside a subset of Hausdorff dimension $1$. As described in \cite{canary_epstein_green2006}, the hyperbolic metrics on the totally geodesic pieces "merge" together, defining a complete hyperbolic metric $m$ on $\partial \CC M$. The locus where the boundary of the convex core is not flat is a geodesic lamination $\lambda$, i. e. a closed subset that is union of disjoint simple geodesics. The surface $\partial \CC M$ is bent along $\lambda$, and the amount of bending can be described by a measured lamination $\mu$, called the \emph{bending measure} of $\partial \CC M$. The $\mu$-\hsk measure along an arc $k$ transverse to $\lambda$ consists of an integral sum of the exterior dihedral angles along the leaves that $k$ meets. By locally integrating the lengths of the leaves of the lamination in $\dd{\mu}$, we obtain the notion of length of the bending measure with respect to the hyperbolic structure $m$, which will be denoted by $\length_m(\mu)$. For a more detailed description we refer to \cite[Section~II.1.11]{canary_epstein_green2006}, or \cite{bonahon1988the_geometry}.
 
\subsection{Incompressible boundary}

When $M$ is convex co-\hsk compact and $\Gamma$ is a discrete and torsion-\hsk free subgroup of isometries of $\Hyp^3$ such that $M \cong \Hyp^3 / \Gamma$, $\Gamma$ acts freely and properly discontinuous on the domain of discontinuity $\Omega_\Gamma$, and the quotient of $\Hyp^3 \cup \Omega_\Gamma$ bt $\Gamma$ determines a natural compactification of $M$, which will be denoted by $\overline{M} = M \cup \partial_\infty M$. Then $M$ is said to have \emph{incompressible boundary} if the inclusion $S \rightarrow \overline{M}$ of each connected component $S$ of $\partial_\infty M$ induces an injection at the level of the fundamental groups. This implies in particular that any lift of the inclusion $S \rightarrow \overline{M}$ to the universal covers $\widetilde{S} \rightarrow \widetilde{\overline{M}}$ is a homeomorphism onto its image.

\subsection{Constant Gaussian curvature surfaces}

\begin{definition}
	Let $S$ be an immersed surface inside a Riemannian $3$-\hsk manifold $N$. The \emph{first fundamental form} $\I$ of $S$ is the Riemannian metric of $S$ given by the restriction of the metric of $N$ to the tangent spaces of $S$. If $S$ admits a unitary normal vector field $\mappa{\nu}{S}{T^1 N}$, we define its \emph{shape operator} $B$ to be the endomorphism of $T S$ given by $B U \defin - \altmathcal{D}_U \nu$, for every tangent vector field $U$ of $S$ (here $\altmathcal{D}$ denotes the Levi-\hsk Civita connection of $N$). The trace of the shape operator will be called the \emph{mean curvature} of $S$, and the tensor $\II \defin \I(B \cdot, \cdot)$ the \emph{second fundamental form} of $S$.
\end{definition}

Let $\Sigma$ be a surface immersed in a hyperbolic $3$-\hsk manifold $M$, with first and second fundamental forms $\I$ and $\II$, and shape operator $B$. We denote by $K_e$ its \emph{extrinsic curvature}, i. e. $K_e = \det B$, and by $K_i$ its \emph{intrinsic curvature}, i. e. the Gaussian curvature of the Riemannian metric $\I$. Then the Gauss-\hsk Codazzi equations of $(\Sigma,\I,\II)$ can be expressed as follows:
\begin{gather*}
    K_i =  K_e - 1 , \\
    (\nabla_U B)V = (\nabla_V B)U \quad \forall U, V ,
\end{gather*}
where $U$ and $V$ are tangent vector fields to $\Sigma$, and $\nabla$ is the Levi-Civita connection of the metric $\I$.

\begin{definition}
    Let $\Sigma$ be an immersed surface inside a hyperbolic $3$-\hsk manifold, and let $k \in (-1,0)$. $\Sigma$ is a \emph{$k$-\hsk surface} if its intrinsic curvature is constantly equal to $k$.
\end{definition}

If $\Sigma$ is a $k$-\hsk surface, then its extrinsic curvature $K_e = k + 1$ is positive, since $k \in (-1,0)$. In particular, $\Sigma$ is a (locally) strictly convex surface.

In every convex co-\hsk compact $3$-\hsk manifold $M$, the subset $M \setminus \CC M$ is the disjoint union of a finite number of geometrically finite hyperbolic ends $(E_i)_i$, each of which is homeomorphic to $\Sigma_i \times (0, \infty)$, for some compact orientable surface $\Sigma_i$ of genus larger than or equal to $2$. By the work of Labourie \cite{labourie1991probleme}, the sets $E_i$ are foliated by embedded $k$-\hsk surfaces $(\Sigma_{i,k})_k$, with $k$ that varies in $(-1,0)$. The surfaces $\Sigma_{i,k}$ approach the components of the pleated boundary $\partial \CC M$ of the convex core of $M$ as $k$ goes to $-1$, and the components of conformal boundary at infinity $\partial_\infty M$ as $k$ goes to $0$. 

We will denote by $M_k$ the compact region of $M$ whose boundary $\partial M_k$ consists of the union of the surfaces $\bigcup_i \Sigma_{i,k}$, and we will endow $\partial M_k$ with the second fundamental form $\II_k$ defined by the normal vector field pointing towards $\partial M_k$, so that $\II_k$ is positive definite, and $H_k$ is a positive function (observe that the eigenvalues of the shape operator have the same sign since $K_e = \det B > 0$).

\subsection{Deformation spaces} \label{subsec:deform_spaces}
Let $\Sigma$ be a compact orientable surface of genus larger than or equal to $2$. The \emph{Teichm\"uller space} of $\Sigma$, denoted by $\Teich(\Sigma)$, is the space of isotopy classes of hyperbolic metrics on $\Sigma$. Equivalently, in light of the Uniformization Theorem, $\Teich(\Sigma)$ can be described as the space of isotopy classes of conformal structures over $\Sigma$ (compatible with the choice of a fixed orientation on $\Sigma$).

Since convex co-\hsk compact hyperbolic $3$-\hsk manifolds are not closed, several different notions of deformation spaces can be introduced. In this exposition we will consider the \emph{quasi-\hsk isometric} (or quasi-\hsk conformal) deformation space. 

\begin{definition}
    Given $M$, $M'$ hyperbolic manifolds, a diffeomorphism $M \rightarrow M'$ is a \emph{quasi-\hsk isometric deformation} of $M$ if it globally bi-\hsk Lipschitz. We denote by $\QD(M)$ the space of quasi-\hsk isometric deformations of $M$, where we identify two deformations $M \rightarrow M'$ and $M \rightarrow M''$ if their pullback metrics are isotopic to each other.
\end{definition}

\begin{remark} \label{rmk:quasi_isom_quasi_conf}
    By a Theorem of Thurston \cite[Proposition 8.3.4]{thurston1979geometry}, two hyperbolic $n$-\hsk manifolds $M$ and $M'$ are quasi-\hsk isometric if and only if their fundamental groups $\Gamma$, $\Gamma'$ (as subgroups of the isometry group of $\Hyp^n$) are quasi-\hsk conformally conjugated, i. e. there exists a quasi-\hsk conformal self-\hsk homeomorphism $\varphi$ of $\partial_\infty \Hyp^n$ such that $\varphi \Gamma \varphi^{-1} = \Gamma'$.
\end{remark}

We denote by $m_k(M) \in \Teich(\partial M_k) = \prod_i \Teich(\Sigma_i) $ the isotopy class of the hyperbolic metric $(- k) \, \I_k$, where $\I_k$ is the first fundamental form of the $k$-\hsk surface $\partial_k M$ of $M$. Then for every $k \in (-1,0)$ we have maps
\[
\begin{matrix}
	T_k \vcentcolon & \QD(M) & \longrightarrow & \Teich(\partial M_k) \\
	& M & \longmapsto & m_k(M) . \\
\end{matrix}
\]
The convenience in considering foliations by $k$-\hsk surfaces relies in the following result, based on the works of Labourie \cite{labourie1992metriques} and Schlenker \cite{schlenker2006hyperbolic}:

\begin{theorem} \label{thm:ksurfaces_homeo}
    If $M$ has incompressible boundary the map $T_k$ is a $\mathscr{C}^1$-\hsk diffeomorphism for every $k \in (-1,0)$.
\end{theorem}

In the non-\hsk incompressible case a similar statement can be recovered, replacing the role of the Teichm\"uller space $\Teich(\partial M_k)$ with its quotient by the action of a suitable subgroup of the mapping class group of $\partial M_k$ (see e. g. \cite[Theorem 5.1.3]{marden2007outer} for the corresponding statement concerning the conformal structure of the boundary at infinity).

As mentioned in the introduction, it is an open question, asked by W. P. Thurston, whether the same statement is true for the hyperbolic structures on the boundary of the convex core, which could be considered as the case $k = - 1$ in Theorem \ref{thm:ksurfaces_homeo}. More precisely, the map $T_{-1}$ is known to be continuously differentiable by \cite{bonahon1998variations}, surjective by the work of Sullivan (described in \cite{canary_epstein_green2006}), but there are no results concerning its injectivity.

\subsection{Dual volume} \label{subsec:dual_volume}

Let $M$ be a convex co-compact hyperbolic $3$-\hsk manifold. If $N$ is a compact convex subset of $M$ with smooth boundary, we define the \emph{dual volume} of $N$ to be
    \[
    V^*(N) \defin V(N) - \frac{1}{2} \int_{\partial N} H \dd{a},
    \]
    where $H$ stands for the mean curvature of $\partial N$ defined using the inner normal vector field, and $V(N)$ is the Riemannian volume of $N$. We refer to \cite{mazzoli2020thesis} for a description of the relation between the notion of dual volume and the polarity correspondence between the hyperbolic and de Sitter spaces.
    
    For every $k \in (-1,0)$, we set $\mappa{V_k^*}{\Teich(\partial M_k)}{\R}$ to denote the function that associates, with a hyperbolic structure $m_k \in \Teich(\partial M_k)$, the dual volume of the region $\partial M_k'$ enclosed by the $k$-\hsk surfaces of the unique convex co-\hsk compact hyperbolic $3$-\hsk manifold $M' = T_k^{-1}(m_k)$ whose $k$-\hsk surfaces have hyperbolic structure $m_k$.
    
    \vspace{1em}
    
    If $(N_h)_h$ is a sequence of convex compact subsets approaching $\CC M$, then the integral of the mean curvature over $\partial N_h$ approaches $\length_m(\mu)$, the length of the bending measure $\mu$ with respect to the hyperbolic structure of $\partial \CC M$. This suggests us to set the \emph{dual volume of the convex core} of $M$ as
    \[
    V_\CC^*(M) \defin V(\CC M) - \frac{1}{2} \length_m(\mu) .
    \]
    In \cite{mazzoli2018the_dual} a first order variation formula for the function $V_\CC^*$ over $\QD(M)$ is studied, called the \emph{dual Bonahon-\hsk Schl\"afli formula}:
    \[
    \dd{V_\CC^*}(\dot{M}) = - \frac{1}{2} \dd{L_\mu}(\dot{m}) ,
    \]
    where $\dot{m}$ denotes the first order variation of the hyperbolic metric on $\partial \CC M$ along $\dot{M}$, and $\mappa{L_\mu}{\Teich(\partial \CC M)}{\R}$ is the function that associates with every hyperbolic structure $m$ the length of the $m$-\hsk geodesic realization of $\mu$.
    
    A strong similarity between dual and renormalized volumes is displayed by their variations formulae. The renormalized volume satisfies
    \[
    \dd{V_R}(\dot{M}) = - \frac{1}{2} \dd{\ext_{\altmathcal{F}_\infty}}(\dot{c}_\infty) ,
    \]
    where $\dot{c}_\infty$ denotes the first order variation of the conformal structure on  $\partial_\infty M$ along $\dot{M}$, and $\mappa{\ext_{\altmathcal{F}_\infty}}{\Teich(\partial_\infty M)}{\R}$ is the function that associates with every conformal structure $c$ the extremal length of the horizontal measured foliation of the Schwarzian at infinity of $M$ with respect to $c$ (see Schlenker \cite{schlenker2017notes} for a proof of this relation).
    
\subsection{Norms on \texorpdfstring{$T \Teich(\Sigma)$}{T Teich(S)}} \label{subsec:norms_on_TT}

First we introduce the necessary notation for the ``Riemannian geometric tools'' that will be used in the rest of the paper. Let $(N, g)$ be a Riemannian manifold, and consider $(e_i)_i$ a local $g$-\hsk orthonormal frame. Given $T$ a symmetric $2$-\hsk tensor on $N$, we define its $g$-\hsk divergence as the $1$-\hsk form
\[
(\divr_g T)(X) \defin \sum_i ({}^g \nabla_{e_i} T)(e_i, X) ,
\]
for every tangent vector field $X$. Similarly, the $g$-\hsk divergence of a vector field $X$ is the function
\[
\divr_g X = \sum_i g({}^g \nabla_{e_i} X, e_i) .
\]
The Laplace-\hsk Beltrami operator can be expressed as $\Delta_g f = \divr_g \grd_g f$. Given two symmetric tensors $T$, $T'$, their scalar product is defined as
\[
\scall{T}{T'}_g \defin g^{i j} \, g^{h k} \, T_{i h} \, T_{j k}' = \tr(g^{-1} T \, g^{-1} T') .
\]
In particular, we set $\tr_g T \defin \scall{g}{T}_g = \tr(g^{-1} T)$. In the next sections it will also be useful to keep in mind the way that these operators change if with replace $g$ with $\lambda g$, for some positive constant $\lambda$:
\begin{gather}
    \divr_{\lambda g} T = \lambda^{-1} \, \divr_g T, \quad \Delta_{\lambda g} f = \lambda^{-1} \, \Delta_g f, \quad \dd{a}_{\lambda g} = \lambda^{n/2} \, \dd{a}_g , \label{eq:1} \\
    \scall{T}{T'}_{\lambda g} = \lambda^{-2} \scall{T}{T'}_g, \qquad \tr_{\lambda g} T = \lambda^{-1} \tr_g T , \label{eq:2}
\end{gather}
if $\dim N = n$.

\vspace{1em}

Let now $\altmathcal{M}$ be the set of Riemannian metrics on $\Sigma$, and let $\altmathcal{H}$ be the subset of the hyperbolic ones. The first order variations $\dot{g}$ of elements of $\altmathcal{M}$ identify with smooth symmetric $2$-\hsk tensors on $\Sigma$. The choice of a metric $g \in \altmathcal{M}$ determines a scalar product on $T_g \altmathcal{M}$, which can be expressed as
\[
\scall{\sigma}{\tau}_{\textit{FT},g} \defin \int_{\Sigma} \scall{\sigma}{\tau}_g \dd{a}_g ,
\]
where \textit{FT} stands for Fischer-Tromba. We define $S^{tt}_2(\Sigma,g)$ to be the space of those symmetric tensors $\sigma$ that are traceless with respect to $g$ (i. e. $\scall{\sigma}{g}_g = 0$) and $g$-\hsk divergence-\hsk free (i. e. $\divr_g \sigma = 0$, as defined above). Such tensors are also called \emph{transverse traceless}. A simple way to characterize the space $S^{tt}_2(\Sigma, g)$ is through \emph{holomorphic quadratic differentials}. A holomorphic quadratic differential $\phi$ on $(\Sigma, g)$ is a $\C$-\hsk valued symmetric tensor that can be locally written as $\phi = f \dd{z}^2$, where $z$ is a local coordinate conformal to the metric $g$ (and compatible with a given orientation), and $f = f(z)$ is a holomorphic function. Transverse traceless tensors are exactly those $2$-\hsk tensors that can be written as $\Re \phi$, for some $\phi$ holomorphic quadratic differential on $(\Sigma, h)$.

It is shown in \cite{tromba2012teichmuller} that, for every hyperbolic metric $h$, $S^{tt}_2(\Sigma,h)$ coincides with
\[
T_h \altmathcal{H} \cap (T_h (\mathrm{Diff}_0(\Sigma) \cdot h))^\perp ,
\]
where $T_h (\mathrm{Diff}_0(\Sigma) \cdot h)$ is the tangent space to the orbit of $h$ by the action of the group of diffeomorphisms of $\Sigma$ isotopic to the identity, and $(\cdot)^\perp$ is taken with respect to the scalar product $\scall{\cdot}{\cdot}_{\textit{FT}, h}$ on $T_h \altmathcal{M}$. Therefore, if $m = [h]$ denotes the isotopy class of a hyperbolic metric on $\Sigma$, we can identify $S^{tt}_2(\Sigma,h)$ with $T_m \Teich(\Sigma)$, the tangent space at $m$ to Teichm\"uller space $\Teich(\Sigma) = \altmathcal{H} / \mathrm{Diff}_0(\Sigma)$, seen as the space of isotopy classes of hyperbolic metrics on $\Sigma$. Moreover, the restriction of the scalar product $\scall{\cdot}{\cdot}_{\textit{FT},h}$ to $S^{tt}_2(\Sigma,h)$ coincides with (a multiple of) the \emph{Weil-\hsk Petersson metric} $\scal{\cdot}{\cdot}_\WP$ (see Lemma \ref{lemma:relations_between_norms} for the explicit multiplicative constant).

\vspace{1em}

The Teichm\"uller space can also be endowed with another Finsler norm that arises from its conformal (or quasi-conformal) interpretation, namely the \emph{Teichm\"uller norm}. The Teichm\"uller norm $\norm{\cdot}_\Teich$ of a tangent vector $\dot{m} \in T_{m} \Teich(\Sigma)$ is the infimum of the $L^\infty$-\hsk norms of the Beltrami differentials representing $\dot{m}$. It is not difficult to see that the Beltrami differential associated to the tangent direction $2 \Re \phi$ coincides with $\nu_{\phi}$, the \emph{harmonic Beltrami differential} associated to $\phi$ (see e. g. \cite{gardiner2000quasiconformal} for a detailed description of these notions, and \cite[Lemma 1.2]{mazzoli2019dual_volume_WP} for a direct computation of this relation). Moreover, the $L^\infty$-\hsk norm of $\nu_{\phi}$ can be computed as follows
\[
\norm{\nu_{\phi}}_\infty = \frac{1}{\sqrt{2}} \sup_{\Sigma} \norm{\Re \phi}_{h} .
\]

We summarize what we observed in the following Lemma:

\begin{lemma} \label{lemma:relations_between_norms}
    For every hyperbolic metric $h$ representing the isotopy class $m \in \Teich(\Sigma)$, the tangent space $T_m \Teich(\Sigma)$ identifies with $S^{tt}_2(\Sigma,h)$. For every $\dot{m} \in T_m \Teich(\Sigma)$ we have
    \begin{gather*}
        \norm{\dot{m}}_\WP = \frac{1}{\sqrt{2}} \norm{\Re \phi}_{\textit{FT}, h} , \\
        \norm{\dot{m}}_\Teich \leq \frac{1}{\sqrt{2}} \sup_\Sigma \norm{\Re \phi}_{h} ,
    \end{gather*}
    where $\phi$ is a holomorphic quadratic differential such that $2 \Re \phi$ represents $\dot{m}$ inside $S^{tt}_2(\Sigma,h)$.
\end{lemma}

\section{Some useful estimates} \label{sec:some_useful_estimates}

In this section we determine estimates for the solution $u_k$ of a certain linear PDE, defined over a $k$-\hsk surface lying inside an end of a convex co-\hsk compact hyperbolic $3$-\hsk manifold with incompressible boundary. The function $u_k$ will be later used to describe the Weil-\hsk Petersson gradient of the dual volume functions $V_k^*$, and the bounds produced in this section will play an important role in the study of its flow. 

\vspace{1em}

Given $(N, g)$ a Riemannian manifold with Levi-\hsk Civita connection ${}^g \nabla$ and area form $\dd{a}_g$, we denote by $H^n(N, \dd{a}_g)$ the Sobolev space of real-\hsk valued functions $f$ on $N$ with $L^2(N, \dd{a}_g)$-\hsk integrable weak derivatives $({}^g \nabla)^i f$ for all $i \leq n$. The space $H^n(N, \dd{a}_g)$ is Hilbert if endowed with the scalar product
\[
\scall{f}{f'} \defin \sum_{i = 0}^n \int_N \scall{({}^g \nabla)^i f}{({}^g \nabla)^i f'}_g \dd{a_g} , \qquad f, f' \in H^n(N, \dd{a}_g) ,
\]
where $\scall{\cdot}{\cdot}_g$ denotes the scalar product induced by $g$ on the space of $i$-\hsk tensors over $N$. Given $\mappa{f}{N}{\R}$ a $\mathscr{C}^n$-\hsk function, we define its $\mathscr{C}^n(N, g)$-\hsk norm as
\[
\norm{f}_{\mathscr{C}^n(N, g)} \defin \sum_{i = 0}^n \sup_{p \in N} \norm{\left. ({}^g \nabla)^i f \right|_p}_g ,
\]
where $\norm{T}_g = \sqrt{\scall{T}{T}_g}$.

Let now $h_k$ denote the hyperbolic metric $(- k) \I_k$ on the $k$-\hsk surface $\partial M_k$, with Levi-\hsk Civita connection ${}^k \nabla$ and Laplace-\hsk Beltrami operator $\Delta_k$ (here we consider $\Delta_k u$ to be the trace of the Hessian of $u$). We define the linear differential operator $L_k$ to be
\[
L_k u \defin (\Delta_k - 2 \, \1) u = \Delta_k u - 2 u .
\]
Let $A$ be the symmetric bilinear form on $H^1(\partial M_k, \dd{a_k})$ with quadratic form
\[
A(u,u) \defin - \scall{L_k u}{u} = \int_\Sigma (\norm{\dd{u}}_k^2 + 2 u^2) \dd{a_k} ,
\]
where $\norm{\cdot}_k$ and $\dd{a_k}$ denote the norm and the area form of $h_k$, respectively. By the Lax-\hsk Milgram's theorem (see e. g. \cite[Corollary 5.8]{brezis2011functional}) applied to the Sobolev space $H^1(\partial M_k, \dd{a_k})$ and to the coercive symmetric bilinear form $A$ we have that, for every $f \in L^2(\partial M_k, \dd{a_k})$, there exists a unique weak solution $u \in H^1(\partial M_k,\dd{a_k})$ of the equation $L_k u = f$. We will in particular denote by $u_k$ the function satisfying
\begin{equation} \label{eq:pde}
    L_k u_k = - k^{-1} H_k \Leftrightarrow \Delta_{\I_k} u_k + 2 k u_k = H_k ,
\end{equation}
where $H_k$ denotes the mean curvature of the $k$-\hsk surface $\partial M_k$. By the classical regularity theory for linear elliptic PDE's (see e. g. \cite[Section 6.3]{evans1998partial}), the smoothness of the mean curvature $H_k$ and the compactness of $\partial M_k$ imply that the function $u_k$ is smooth and it is a strong solution of equation \eqref{eq:pde}. 

By the work of Rosemberg and Spruck \cite[Theorem 4]{rosenberg1994existence}, for every Jordan curve $c$ in $\partial_\infty \Hyp^3$ there exist exactly two $k$-\hsk surfaces $\widetilde{\Sigma}_k^\pm(c)$ asymptotic to $c$. A fundamental property of $k$-\hsk surfaces, which will crucial in Lemma \ref{lem:estimates_uk}, is described by the following Proposition.

\begin{proposition}[{\cite[Proposition 3.8]{bonsante2019induced}}] \label{prop:bounds_mean_curvature}
    Let $k \in (-1,0)$ and $n \in \N$. Then there exists a constant $N_{k,n} > 0$ such that, for every Jordan curve $c$ in $\partial_\infty \Hyp^3$, the mean curvature $H_{c,k}$ of the $k$-\hsk surface $\widetilde{\Sigma}_k(c) = \widetilde{\Sigma}_k^+(c) \sqcup \widetilde{\Sigma}_k^-(c)$ asymptotic to $c$ satisfies
    \[
    \norm{H_{c,k}}_{\mathscr{C}^n(\widetilde{\Sigma}_k(c))} \leq N_{n, k} .
    \]
\end{proposition}

\begin{proof}
    We briefly recall here the proof of this statement (which was stated in \cite{bonsante2019induced} for $n = 0$). $k$-\hsk surfaces satisfy the following compactness criterion:
    \begin{proposition}[{\cite[Proposition 3.6]{bonsante2019induced}}] \label{prop:compactness_ksurfaces}
        Let $k \in (-1,0)$, and consider $\mappa{f_n}{\Hyp^2_k}{\Hyp^3}$ a sequence of proper isometric embeddings of the hyperbolic plane $\Hyp^2_k$ with constant Gaussian curvature $k$. If there exists a point $p \in \Hyp^2$ such that $(f_n(p))_n$ is precompact, then there exists a subsequence of $(f_n)_n$ that converges $\mathscr{C}^\infty$-\hsk uniformly on compact sets to an isometric immersion $\mappa{f}{\Hyp^2_k}{\Hyp^3}$.
    \end{proposition}
    
    Fixed $k \in (-1,0)$ and $n \in \N$, assume by contradiction that there exists a sequence of Jordan curves $(c_m)_m$ such that the mean curvatures $H_m = H_{c_m, k}$ of the $k$-\hsk surfaces $\widetilde{\Sigma}_k(c_m)$ satisfy $\norm{H_m}_{\mathscr{C}^n(\widetilde{\Sigma}_k(c_m))} > m$. Up to extracting a subsequence, there exists an $i \leq n$ such that for every $m \in \N$
    \[
    \sup_{\widetilde{\Sigma}_k(c_m)} \norm{({}^k \nabla)^i H_m} > \frac{m}{n + 1} = C_n \, m.
    \]
    Now choose $q_m \in \widetilde{\Sigma}_k(c_m)$ for which the norm of $({}^k \nabla)^i H_m$ at $q_m$ is $\geq C_n \,m$. Since each component of $\widetilde{\Sigma}_k(c_m)$ is embedded and isometric to the hyperbolic plane $\Hyp^2_k$ (which is homogeneous), we can find a sequence of proper isometric embeddings $\mappa{f_m}{\Hyp^2_k}{\Hyp^3}$, parametrizing a component of $\widetilde{\Sigma}_k(c_m)$, such that $f_m(\bar{p}) = q_m$ for some fixed basepoint $\bar{p} \in \Hyp^2_k$. Up to post-\hsk composing $f_m$ by an isometry of $\Hyp^3$, we can also assume that $f_m(\bar{p}) = \bar{q}$ is fixed. In this way, we have found a sequence of proper isometric embeddings $\mappa{f_m}{\Hyp^2_k}{\Hyp^3}$ satisfying
    \begin{itemize}
        \item $f_m(\bar{p}) = \bar{q} \in \Hyp^3$ is independent of $m \in \N$;
        \item the mean curvature of the surfaces $f_m(\Hyp^2_k)$ at $\bar{q}$ has some $i$-\hsk th order derivative that is unbounded as $m$ goes to $\infty$.
    \end{itemize}
    This clearly contradicts the compactness criterion mentioned above.
\end{proof}

From this result we can now obtain a uniform control on $u_k$:

\begin{lemma} \label{lem:estimates_uk}
    Let $M$ be a convex co-\hsk compact hyperbolic $3$-\hsk manifold. Then the function $\mappa{u_k}{\partial M_k}{\R}$, solution of \eqref{eq:pde}, satisfies
    \[
    \frac{\max_{\partial M_k} H_k}{2 k} \leq u_k \leq \frac{\min_{\partial M_k} H_k}{2 k} = \frac{\sqrt{k + 1}}{k} < 0.
    \]
    Moreover, if $M$ has incompressible boundary, then there exists a constant $C_k > 0$ depending only on the intrinsic curvature $k \in (-1,0)$, and in particular not on the hyperbolic structure of $M$, such that
    \[
    \max_{\partial M_k} \norm{{}^k \nabla^2 u_k}_k \leq C_k .
    \]
\end{lemma}

\begin{proof}
    The first assertion is an immediate consequence of the maximum principle applied to $u_k$ as a solution of the PDE \eqref{eq:pde}. Moreover, since the product of the principal curvatures (i. e. the eigenvalues of the shape operator) of a $k$-\hsk surface is everywhere equal to $k + 1$, the trace of the shape operator is bounded from below by $2 \sqrt{k + 1}$ (see Remark \ref{rmk:about_the_bounds} for an explanation of the equality $\min_{\partial M_k} H_k = 2 \sqrt{k + 1}$). 
    
    The proof of the second part of the assertion requires more care. Let $\Sigma_k$ be a connected component of the $k$-\hsk surface $\partial M_k$, and let $\widetilde{M} \cong \Hyp^3$ denote the universal cover of $M$. Since $M$ is a convex co-\hsk compact hyperbolic $3$-\hsk manifold with incompressible boundary, every component $\widetilde{\Sigma}_k$ of the preimage of $\Sigma_k$ in $\widetilde{M}$ is stabilized by a subgroup $\Gamma \cong \pi_1(\Sigma_k)$ of the fundamental group of $M$, acting by isometries on $\widetilde{M}$. Each of these subgroups $\Gamma$ is quasi-\hsk Fuchsian (see e. g. \cite[Corollary 4.112 and Theorem 8.17]{kapovich2001hyperbolic} for a proof of this assertion), and the surface $\widetilde{\Sigma}_k$ is a $k$-\hsk surface asymptotic to some Jordan curve in $\partial_\infty \widetilde{M} \cong \partial_\infty \Hyp^3$. In particular, by Proposition \ref{prop:bounds_mean_curvature}, we can find a universal constant $N_k = N_{2,k} > 0$ that satisfies
    \begin{equation} \label{eq:control_mean_curvature}
        \norm*{\tilde{H}_k}_{\mathscr{C}^2(\widetilde{\Sigma}_k)} \leq N_k .
    \end{equation}
    Here we stress that the constant $N_k$ does not depend on the hyperbolic structure of $M$, or $\Sigma_k$, but only on the value of $k \in (-1,0)$.
    
    Our goal is now to make use of this control to obtain a uniform bound of the norm of the Hessian of $u_k$. For this purpose, we will need the following classical result of regularity for linear elliptic differential equations:
    
    \begin{theorem}[{\cite[Theorem 2, page 314]{evans1998partial}}] \label{thm:regularity_pde}
        Let $m, n \in \N$ and $U \subset \R^n$ a bounded open set. We consider a differential operator $L$ of the form:
        \[
        L f = - \sum_{i, j = 1}^n a^{i j}(x) \, \partial^2_{x_i, x_j} f + \sum_{i = 0}^n b^i(x) \, \partial_{x_i} f + c(x) f ,
        \]
        where $a^{i j} = a^{j i}, b^i, c \in \mathscr{C}^{m + 1}(U,\R)$. Assume that $L$ is uniformly elliptic, i. e. there exists a constant $\varepsilon > 0$ such that $\sum_{i, j} a^{i j}(x) v_i v_j \geq \varepsilon \norm{v}^2$ for all $v \in \R^n$ and $x \in U$. If $f \in H^1(U)$ is a weak solution of the equation $L f = \lambda$, for some $\lambda \in H^m(U)$, then for every bounded open set $V$ with closure contained in $U$, there exists a constant $C$, depending only on $m$, $U$, $V$ and the functions $a^{i j}, b^i, c$, such that
        \[
        \norm{f}_{H^{m + 2}(V)} \leq C (\norm{\lambda}_{H^m(U)} + \norm{f}_{L^2(U)}) ,
        \]
        where the Sobolev spaces $H^{m + 2}(V)$, $H^m(U)$, and $L^2(U)$ are defined with respect to the Euclidean metric of $U \subset \R^n$.
    \end{theorem}
	
    The surface $\widetilde{\Sigma}_k$ endowed with the lift of the hyperbolic metric $h_k$ of $\Sigma_k$ is isometric to the hyperbolic plane $\Hyp^2$. In the rest of the proof, we will identify $\widetilde{\Sigma}_k$ with the Poincar\'e disk model $\Hyp^2 \defin (B_1, g)$, where $B_1$ is the Euclidean ball of radius $1$ and center $0$ in $\C$, and $g$ is the Riemannian metric
    \[
    g = \left( \frac{2}{1 - \abs{z}^2} \right)^2 \abs{\dd{z}}^2 .
    \]
    Now we choose $U$ and $V$ to be the $g$-\hsk geodesic balls of center $0 \in B_1$ and hyperbolic radius equal to $2$ and $1$, respectively. The lift of the operator $- L_k$ over $U$ is clearly uniformly elliptic, because of the compactness of $\overline{U}$ and its expression in coordinates:
    \[
    - L_k f = - g^{i j} ( \partial_{i j}^2 f - \Gamma_{i j}^h(g) \, \partial_h f) + 2 f ,
    \]
    where $\Gamma_{i j}^h(g)$ denote the Christoffel symbols of $g$. Again by the compactness of $\overline{U}$ and $\overline{V}$, the norms of the Sobolev spaces $\norm{\cdot}_{H^j(U)}$ and $\norm{\cdot}_{H^j(V)}$, computed with respect to the flat connection of $B_1 \subset \R^2$ and the Euclidean volume form, are equivalent to the norms of the corresponding Sobolev spaces defined using the Levi-\hsk Civita connection of $g$ and the $g$-\hsk volume form. Moreover, the bi-\hsk Lipschitz constants involved in the equivalence only depend on a bound of the $\mathscr{C}^{j + 1}$-\hsk norm of $g$ over $U$, therefore they can be chosen to depend only on $j \in \N$. From now on, we will always consider the norms on the spaces $H^j(U)$ and $H^j(V)$ to be defined using the metric $g$ and its connection. 
    
    Now we apply Theorem \ref{thm:regularity_pde} to $m = n = 2$, the operator $- L_k$ and the functions $f = \tilde{u}_k$, $\lambda = - k^{-1} \tilde{H}_k$, where $\tilde{F}$ denotes the lift of the function $F$ over $\widetilde{\Sigma}_k$: we can find a universal constant $C > 0$ (depending only on the open sets $U$, $V$, that we chose once for all, and on the metric $g |_U$) such that:
    \[
    \norm{\tilde{u}_k}_{H^4(V)} \leq C ( - k^{-1} \norm*{\tilde{H}_k}_{H^2(U)} + \norm{\tilde{u}_k}_{L^2(U)} ) .
    \]
    By the first part of Lemma \ref{lem:estimates_uk}, $\norm{\tilde{u}_k}_{\mathscr{C}^0(U)} \leq - (2 k)^{-1} \norm*{\tilde{H}_k}_{\mathscr{C}^0(\Hyp^2)}$. In addition, we have
    \[
    \norm{\tilde{u}_k}_{L^2(U)} \leq \Area(U,g)^{1/2} \norm{\tilde{u}_k}_{\mathscr{C}^0(U)} \leq - (2 k)^{-1} \Area(U,g)^{1/2} \norm*{\tilde{H}_k}_{\mathscr{C}^0(\Hyp^2)} ,
    \]
    and 
    \[
    \norm*{\tilde{H}_k}_{H^2(U)} \leq \Area(U,g)^{1/2} \norm*{\tilde{H}_k}_{\mathscr{C}^2(\Hyp^2)} .
    \]
    In conclusion, we deduce that
    \[
    \norm{\tilde{u}_k}_{H^4(V)} \leq - 2 k^{-1} C \Area(U, g)^{1/2} \norm*{\tilde{H}_k}_{\mathscr{C}^2(\Hyp^2)} .
    \]
    By the Sobolev embedding theorem (see e. g. \cite[Corollary 9.13, page 283]{brezis2011functional}), given $W$ an open set satisfying $0 \in W \subset \overline{W} \subset V$, the $\mathscr{C}^2(W)$-\hsk norm of $\tilde{u}_k$ (again, computed with respect to the Levi-\hsk Civita connection of $g$) is controlled by a multiple of its $H^4$-\hsk norm over $V$,  and the multiplicative factor depends only on $W$ and $V$. Therefore, if we choose for instance $W = B_{\Hyp^2}(0,1/2)$ we get:
    \[
    \norm*{{}^k \nabla^2 \tilde{u}_k}_{\mathscr{C}^0(W)} \leq C'(k) \, \norm*{\tilde{H}_k}_{\mathscr{C}^2(\Hyp^2)} .
    \]
    Now the desired statement easily follows. From relation \eqref{eq:control_mean_curvature} and the last inequality, we obtain a uniform bound of the Hessian of $\tilde{u}_k$ over $W \ni 0$. Let now $q$ be any other point of $\Hyp^2$, and choose a $g$-\hsk isometry $\mappa{\varphi_q}{B_1}{B_1}$ such that $\varphi_q(0) = q$. If we replace $\tilde{u}_k$ and $\tilde{H}_k$ with $\tilde{u}_k \circ \varphi_q$ and $\tilde{H}_k \circ \varphi_q$, respectively, the exact same argument above applies, since the operator $L_k$ and the norms $\norm{\cdot}_{H^j}$, $\norm{\cdot}_{\mathscr{C}^l}$ are invariant under the action of the isometry group of $\Hyp^2$ (and since $\norm*{\tilde{H}_k}_{\mathscr{C}^2(\Hyp^2)} = \norm*{\tilde{H}_k \circ \varphi_q}_{\mathscr{C}^2(\Hyp^2)}$ ). In particular, this gives us a control of the norm of ${}^k \nabla^2 \tilde{u}_k$ over $\varphi_q(W)$ for any point $q \in \Hyp^2$, and the last part of our assertion follows.
\end{proof}

\begin{remark} \label{rmk:about_the_bounds}
    The minimum of the mean curvature $2 \sqrt{k + 1}$ is always realized. As described by Labourie in \cite{labourie1992surfaces}, whenever we have a $k$-\hsk surface $\Sigma_k$ with first and second fundamental forms $\I_k$ and $\II_k$, respectively, the identity map $\mappa{\id}{(\Sigma_k, \II_k)}{(\Sigma_k, \I_k)}$ is harmonic, with Hopf differential $\psi_k$ satisfying
    \[
    2 \Re \psi_k = \I_k - \frac{H_k}{2(k + 1)} \II_k .
    \]
    Its squared norm with respect to $\II_k$ can be expressed as follows
    \[
    \norm{2 \Re \psi_k}_{\II_k}^2 = \frac{H_k^2 - 4(k + 1)}{(k + 1)^2} .
    \]
    In particular, at each zero of $\psi_k$ (which necessarily exist because $\chi(\Sigma_k) < 0$) we have $H_k = 2 \sqrt{k + 1}$.
    
    We stress that, even if the maximum of the mean curvature $H_k$ will clearly depend on the hyperbolic structure of $M$, Proposition \ref{prop:bounds_mean_curvature} guarantees that $\max H_k$ is controlled by a function of $k$ independent on the geometry of $M$, as long as $\partial M$ is incompressible.
    
    We will make use of the upper bound $u_k \leq \frac{\sqrt{k + 1}}{k}$ in Lemma \ref{lem:lower_bound_gradient}, where we will determine a lower bound of the Weil-\hsk Petersson norm of the differential of $V_k^*$ in terms of the integral of the mean curvature.
\end{remark}

\section{The gradient of the dual volume} \label{sec:gradient_dual_volume}

The aim of this section is to describe the gradient of the dual volume function $V_k^*$ with respect to the Weil-\hsk Petersson metric on the Teichm\"uller space of $\partial M_k$ in terms of the function $u_k$ studied in the previous section. 

\vspace{1em}

The first order variation of the dual volume of $M_k$ as we vary the convex co-\hsk compact hyperbolic structure of $M$ can be computed applying the \emph{differential Schl\"afli formula} due to Rivin and Schlenker \cite{schlenker_rivin2000schlafli}. In particular, we have:

\begin{proposition} \label{prop:dual_diiferential_schlafli}
    \begin{align*}
        \dd{(V_k^* \circ T_k)}(\dot{M}) & = \frac{1}{4} \int_{\partial M_k} \scall{\dot{\I}_k}{\II_k - H_k \, \I_k}_{\I_k} \dd{a_{\I_k}} \\
        & = \frac{1}{4} \int_{\partial M_k} \scall{\dot{h}_k}{\II_k + k^{-1} H_k \, h_k}_{h_k} \dd{a_{h_k}} ,
    \end{align*}
    where $\dot{I}_k = - k^{-1} \dot{h}_k$ is the first order variation of the first fundamental form on $\partial M_k$ along the variation $\dot{M}$, and $\mappa{T_k}{\QD(M)}{\Teich(\partial M_k)}$ is the diffeomorphism introduced in Section \ref{subsec:deform_spaces}.
\end{proposition}

A proof of this relation based on the result of Rivin and Schlenker can be found in \cite[Proposition 2.5]{mazzoli2018the_dual}. From its variation formula, we can give an explicit description of the Weil-\hsk Petersson gradient of the dual volume function $V_k^*$, which will turn out to be useful for the study of its flow.

\begin{proposition} \label{prop:gradient_dual_volume}
    The vector field $\grd_\WP V_k^*$ is represented by the symmetric $2$-\hsk tensor $2 \Re \phi_k$, where $\phi_k$ is the (unique) holomorphic quadratic differential satisfying
    \[
    \Re \phi_k = \II_k - {}^k \nabla^2 u_k + u_k \, h_k ,
    \]
    where $u_k$ denotes the solution of equation \eqref{eq:pde}.
\end{proposition}

\begin{proof}
    Let $\dot{m}_k$ denote a tangent vector to the Teichm\"uller space of $\partial M_k$ at $m_k$. As described in Section \ref{subsec:deform_spaces}, given any hyperbolic metric $h_k$ representing the isotopy class $m_k \in \Teich(\partial M_k)$, we can find a unique transverse traceless tensor $\dot{h}_k \in S_2^{tt}(\Sigma,h_k)$ representing $\dot{m}_k$. Assume for a moment that we can find a decomposition of the symmetric tensor $\II_k + k^{-1} H_k \, h_k$ of the following form:
    \[
    \II_k + k^{-1} H_k h_k = S_{tt} + \Dlie_X h_k + \lambda \, h_k ,
    \]
    where $S_{tt}$ is a transverse traceless tensor with respect to $h_k$, $\Dlie_X h_k$ is the Lie derivative of $h_k$ with respect to a vector field $X$, and $\lambda$ is a smooth function on $\partial M_k$. Then, by Proposition \ref{prop:dual_diiferential_schlafli}, we could express the variation of the dual volume $V_k^*$ along a transverse traceless variation $\dot{h}_k$ as follows:
    \[
    \dd{V_k^*}(\dot{h}_k) = \frac{1}{4} \int_{\partial M_k} \scall{\dot{h}_k}{S_{tt} + \Dlie_X h_k + \lambda \, h_k}_{h_k} \dd{a_{h_k}} .
    \]
    Since $\dot{h}_k$ is traceless, the scalar product $(\dot{h}_k, h_k)_{h_k} = \tr_{h_k}(\dot{h}_k)$ vanishes identically. Moreover, the $L^2$-\hsk scalar product between $\dot{h}_k$ and $\Dlie_X h_k$ vanishes too, because $\Dlie_X h_k$ is tangent to the orbit of $h_k$ by the action of $\mathrm{Diff}_0(\Sigma)$ (see Section \ref{subsec:norms_on_TT}). In particular, we must have
    \[
    \dd{V_k^*}(\dot{h}_k) = \frac{1}{4} \int_{\partial M_k} \scall{\dot{h}_k}{S_{tt}}_{h_k} \dd{a_{h_k}} = \frac{1}{8} \scall{\dot{h}_k}{2 S_{t t}}_{\textit{FT},h_k}.
    \]
    In light of Lemma \ref{lemma:relations_between_norms}, by varying the tangent vector $\dot{m}_k \in T_{m_k} \Teich(\partial M_k)$, we deduce that the tensor $2 S_{tt}$ is the element of $S^{tt}_2(\Sigma, h_k)$ that represents $\grd_{\WP} V_k^*$.
    
    In conclusion, this argument shows us that, in order to prove our assertion, we need to determine a decomposition of the tensor $\II_k + k^{-1} H_k \, h_k$ of the form we described above, with $S_{tt} = \II_k - {}^k \nabla^2 u_k + u_k \, h_k$. For this purpose, we consider the following expression:
    \begin{align*}
        \II_k + k^{-1} H_k \, h_k & = (\II_k - {}^k \nabla^2 u_k + u_k \, h_k) + {}^k \nabla^2 u_k + ( k^{-1} H_k - u_k ) \, h_k \\
        & = (\II_k - {}^k \nabla^2 u_k + u_k \, h_k) + \frac{1}{2} \Dlie_{\grd_{h_k} u_k} h_k + ( k^{-1} H_k - u_k ) \, h_k ,
    \end{align*}
    where we used the relation $\Dlie_{\grd_{h_k} u_k} h_k = 2 \, {}^k \nabla^2 u_k$. In this expression, the second term of the sum is of the type $\Dlie_X h_k$, while the third term has the form $\lambda \, h_k$. Therefore, by the argument above, it is enough to show that the first term is $h_k$-\hsk traceless and $h_k$-\hsk divergence-\hsk free. The trace of $\II_k - {}^k \nabla^2 u_k + u_k \, h_k$ satisfies
    \[
    \tr_{h_k}(\II_k - {}^k \nabla^2 u_k + u_k \, h_k) = - k^{-1} H_k - \Delta_k u_k + 2 u_k .
    \]
    This expression vanishes because $u_k$ is a solution of equation \eqref{eq:pde}. In order to compute the divergence of our tensor, we will need the following relations:
    \[
    \divr_{h_k} \II_k = - k^{-1} \dd{H_k}, \qquad \divr_g({}^g \nabla^2 f) = \dd(\Delta_g f) + \Ric_g(\grd_g f, \cdot ) .
    \]
    The first equality follows from the Codazzi equation $({}^k \nabla_X B_k)Y = ({}^k \nabla_Y B_k)X$ satisfied by the shape operator $B_k$ of $\partial M_k$ (the Levi-\hsk Civita connections of $h_k$ and the first fundamental form $\I_k$ are the same, since they differ by a multiplicative constant). The second relation is true for any Riemannian metric $g$, and we will apply it in the case $g = h_k$ and $f = u_k$. Since $h_k$ is a hyperbolic metric on a $2$-\hsk manifold, we have $\Ric_{h_k} = - h_k$. Therefore
    \begin{align*}
        \divr_{h_k} (\II_k - \nabla^2_k u_k + u_k \, h_k) & = - k^{-1} \dd{H_k} - \dd(\Delta_k u_k) + \dd{u_k} + \dd{u_k} \\
        & = \dd(- k^{-1} H_k - \Delta_k u_k + 2 u_k ) ,
    \end{align*}
    where we used the relation $\divr_g (f \, g) = \dd{f}$. Again, the expression above vanishes because $u_k$ solves equation \eqref{eq:pde}. Then we have shown that $\II_k - {}^k \nabla^2 u_k + u_k \, h_k$ is a transverse traceless tensor, as desired.
\end{proof}

\begin{remark}
    In fact, the decomposition we presented for the tensor $\II_k + k^{-1} H_k \, h_k$ is related to the orthogonal decomposition of the space of symmetric tensors due to Fischer and Marsden \cite{fischer1975deformations}. Given $g$ a hyperbolic metric, every symmetric tensor $S$ admits an orthogonal decomposition of the following form:
    \[
    S = S_{tt} + \Dlie_X g + ((- \Delta_g f + f) \, g + {}^g \nabla^2 f ) ,
    \]
    where:
    \begin{itemize}
        \item $S_{tt}$ is transverse traceless with respect to $g$;
        \item $S_{tt} + \Dlie_X g$ is tangent to the space of Riemannian metrics with constant Gaussian curvature equal to $-1$. In other words, if $g' \mapsto K(g')$ denotes the operator that associates to the Riemannian metric $g'$ its Gaussian curvature, then $S_{tt} + \Dlie_X g \in \ker \dd{K_g}$;
        \item $(- \Delta_g f + f) \, g + {}^g \nabla^2 f$ lies in the $L^2$-\hsk orthogonal of $\ker \dd{K_g}$.
    \end{itemize}
    Then, the expression
    \begin{align*}
        \II_k + k^{-1} H_k h_k & = (\II_k - {}^k \nabla^2 u_k + u_k \, h_k) + 0 + (( k^{-1} H_k - u_k ) \, h_k + {}^k \nabla^2 u_k) \\
        & = (\II_k - {}^k \nabla^2 u_k + u_k \, h_k) + 0 + (( - \Delta_k u_k + u_k ) \, h_k + {}^k \nabla^2 u_k)
    \end{align*}
    is the Fischer-\hsk Marsden decomposition of $\II_k + k^{-1} H_k \, h_k$, where $f = u_k$, $X = 0$ and $S_{tt} = (\II_k - {}^k \nabla^2 u_k + u_k \, h_k)$.
\end{remark}

Using this explicit description of the Weil-\hsk Petersson gradient of the dual volume function $V_k^*$, we can determine a lower bound of its norm in terms of the integral of the mean curvature:

\begin{lemma} \label{lem:lower_bound_gradient}
    For every $k \in (-1,0)$ we have
    \[
    \norm{\dd{V_k^*}}_\WP^2 \geq - \frac{\sqrt{k + 1}}{2 k} \int_{\partial M_k} H_k \dd{a_{\I_k}} - \frac{2 \pi (k + 1)}{k^2} \abs{\chi(\partial M)} .
    \]
\end{lemma}

\begin{proof}
    In what follows, we will prove the following expression:
    \begin{equation} \label{eq:decomposition_wp_norm}
        \norm{\II_k - \nabla^2_k u_k + u_k \, h_k}_{\I_k}^2 = k u_k H_k - 2(k + 1) + \divr_{\I_k} W,
    \end{equation}
    for some tangent vector field $W$ on $\partial M_k$. Assuming for the moment this relation, we deduce that
    \begin{align*}
        \norm{\dd{V_k^*}}^2_\WP & = \frac{1}{2} \int_{\partial M_k} \norm{\Re \phi_k}_{h_k}^2 \dd{a_{h_k}} \tag{Prop. \ref{prop:gradient_dual_volume} and Lemma \ref{lemma:relations_between_norms}} \\
        & = \frac{1}{2} \int_{\partial M_k} (-k)^{-2} \, \norm{\Re \phi_k}_{\I_k}^2 \, (- k) \dd{a_{\I_k}} \\
        & = - \frac{1}{2 k} \int_{\partial M_k} (k u_k H_k - 2(k + 1)) \dd{a_{\I_k}} \tag{relation \eqref{eq:decomposition_wp_norm}} ,
    \end{align*}
    where we used that $h_k = (- k) \I_k$ and relations \eqref{eq:1}, \eqref{eq:2}, and that the integral of the term $\divr_{\I_k} W$ vanishes by the divergence theorem. By Lemma \ref{lem:estimates_uk}, we have $u_k \leq \frac{\sqrt{k + 1}}{k}$, therefore we obtain
    \[
    \norm{\dd{V_k^*}}^2_\WP \geq - \frac{\sqrt{k + 1}}{2 k} \int_{\partial M_k} H_k \dd{a_{\I_k}} - \frac{2 \pi (k + 1)}{k^2} \abs{\chi(\partial M)} ,
    \]
    where we applied the Gauss-\hsk Bonnet theorem to say that the area of $\partial M_k$ with respect to $\I_k$ is equal to $- 2 \pi k^{-1} \abs{\chi(\partial M)}$.
    
    The only ingredient left to prove is relation \eqref{eq:decomposition_wp_norm}. For this computation, we will use the \emph{Bochner's formula} (see e. g. \cite[page 223]{lee2018introduction}):
    \begin{equation} \label{eq:bochner}
        \frac{1}{2} \Delta_g \norm{\dd{f}}_g^2 = \norm{{}^g \nabla^2 f}_g^2 + g(\grd_g f, \grd_g \Delta_g f) + \Ric_g (\grd_g f, \grd_g f) ,
    \end{equation}
    and the following expressions:
    \begin{gather}
        \divr_g (f X) = g(\grd_g f, X) + f \divr_g X , \label{eq:div1} \\
        \frac{1}{2} \scall{\Dlie_X g}{T}_g = - (\divr_g T)(X) + \divr_g Y , \label{eq:div2}
    \end{gather}
    where $X$ is a tangent vector field, $f$ is a smooth function, $T$ is a symmetric $2$-\hsk tensor, and $Y = T(X, \cdot)^\sharp$ is the vector field defined by requiring that $g(Y,Z) = T(X, Z)$ for all vector fields $Z$. From now on, we will omit everywhere the dependence of the connections, norms, gradients, and the Laplace-\hsk Beltrami operator on the Riemannian metric $g$, and everything has to be interpreted as associated to $g = \I_k$. Observe also that the Levi-\hsk Civita connection of $\I_k$ and $h_k$ are equal, since these metrics differ by the multiplication by a constant and, in particular, the $h_k$- and $\I_k$-\hsk Hessians coincide. Then we have:
    \begin{align}
        \begin{split}
            \norm{\II_k - \nabla^2 u_k + u_k \, h_k}^2 & = \norm{\II_k - \nabla^2 u_k - k\, u_k \, \I_k}^2 \\
            & = \norm{\II_k}^2 + \norm{\nabla^2 u_k}^2 + k^2 \, u_k^2 \, \norm{\I_k}^2 -  2 \scall{\II_k}{\nabla^2 u_k} + \\
            & \qquad \qquad \qquad \qquad - 2 k \, u_k \scall{\II_k}{\I_k} + 2 k \, u_k \scall{\nabla^2 u_k}{\I_k} .
        \end{split} \label{eq:norm_square}
    \end{align}
    First we focus our attention on the terms $\norm{\nabla^2 u_k}^2$ and $\scall{\II_k}{\nabla^2 u_k}$. In order to simplify the notation, we say that two functions $a$ and $b$ on $\partial M_k$ are equal "modulo divergence", and we write $a \equiv_{\divr} b$, if their difference coincides with the divergence of some smooth vector field. Then we have:
    \begin{align*}
        \norm{\nabla^2 u_k}^2 & = \frac{1}{2} \Delta \norm{\dd{u_k}}^2 - \scal{\grd u_k}{\grd \Delta u_k} - k \norm{\dd{u_k}}^2 \tag{relation \eqref{eq:bochner}} \\
        & \equiv_{\divr} - \scal{\grd u_k}{\grd \Delta u_k} - k \norm{\dd{u_k}}^2 \tag{$\Delta_g f = \divr_g \grd_g f$} \\
        & = - \divr(\Delta u_k \grd u_k) + (\Delta u_k)^2 - k \norm{\dd{u_k}}^2 \tag{relation \eqref{eq:div1}} \\
        & \equiv_{\divr} (\Delta u_k)^2 - k \divr(u_k \grd u_k) + k u_k \Delta u_k \tag{relation \eqref{eq:div1}} \\
        & \equiv_{\divr} \Delta u_k (\Delta u_k + k u_k) , 
    \end{align*}
    
    \begin{align*}
        \scall{\II_k}{\nabla^2 u_k} & = \frac{1}{2} \scall{\II_k}{\Dlie_{\grd u_k} \I_k} \tag{$\Dlie_{\grd_g f} \, g = 2 \, {}^g \nabla^2 f$} \\
        & \equiv_{\divr} - (\divr \II_k) (\grd u_k) \tag{relation \eqref{eq:div2}} \\
        & = - \scal{\grd H_k}{\grd u_k} \tag{$\divr \II_k = \dd{H_k}$} \\
        & = - \divr(H_k \grd u_k) + H_k \Delta u_k \tag{relation \eqref{eq:div1}} \\
        & \equiv_{\divr} H_k \Delta u_k .
    \end{align*}
    The other terms in equation \eqref{eq:norm_square} are simpler to handle. In particular we have:
    \begin{align*}
        \norm{\II_k}^2 & = H_k^2 - 2(k + 1) , \\
        \norm{\I_k}^2 & = 2 , \\
        \scall{\II_k}{\I_k} & = H_k , \\
        \scall{\nabla^2 u_k}{\I_k} & = \Delta u_k .
    \end{align*}
    Replacing all the relations we found in equation \eqref{eq:norm_square}, we obtain:
    \begin{align*}
        \norm{\II_k - \nabla^2 u_k + u_k h_k}^2 & \equiv_{\divr} H_k^2 - 2 (k + 1) + \Delta u_k (\Delta u_k + k u_k) + 2 k^2 u_k^2  + \\
        & \qquad \qquad \qquad \qquad - 2 H_k \Delta u_k - 2 k u_k H_k + 2 k u_k \Delta u_k \\
        & = H_k^2 - 2 (k + 1) + 2 k^2 u_k^2 - 2 k u_k H_k +  \\
        & \qquad \qquad \qquad \qquad + \Delta u_k ( \Delta u_k + 3 k u_k - 2 H_k)
    \end{align*}
    Finally, by replacing the expression of $\Delta u_k = \Delta_{\I_k} u_k$ from equation \eqref{eq:pde} in the equality above, we find that:
    \[
    \norm{\II_k - \nabla^2 u_k + u_k h_k}^2 \equiv_{\divr} k u_k H_k - 2(k + 1) ,
    \]
    which is equivalent to relation \eqref{eq:decomposition_wp_norm}.
\end{proof}

Since the Weil-\hsk Petersson metric of the Teichm\"uller space is non-\hsk complete, a control from above of the quantity $\norm{\dd{V_k^*}}_\WP$ would not suffice to guarantee the existence of the flow for every time. For this purpose, we rather study the $L^\infty$-\hsk norm of the Beltrami differentials equivalent to $\grd_\WP V_k^*$, which gives a control with respect to the Teichm\"uller metric (that is complete). At this point, the estimates determined in Lemma \ref{lem:estimates_uk} will play an essential role.

\begin{proposition} \label{prop:bound_teichmuller_norm}
    There exists a constant $D_k > 0$ depending only on the intrinsic curvature $k \in (-1,0)$ such that
    \[
    \norm{\grd_\WP V_k^*}_\Teich \leq D_k ,
    \]
    where $\norm{\cdot}_\Teich$ denotes the Teichm\"uller norm on $T \Teich(\partial M_k)$.
\end{proposition}

\begin{proof}
    Let $m_k$ be a point of the Teichm\"uller space $\Teich(\partial M_k)$ and let $h_k$ be a hyperbolic metric in the isotopy class $m_k$. In Proposition \ref{prop:gradient_dual_volume}, we showed that the vector field $\grd_\WP V_k^*$ at a point $m_k \in \Teich(\partial M_k)$ is represented by the transverse traceless tensor $2 \Re \phi_k \in S^{tt}_2(\partial M_k, h_k)$. 
    Therefore by Lemma \ref{lemma:relations_between_norms} we have
    \[
    \norm{\grd_\WP V_k^*}_\Teich \leq \frac{1}{\sqrt{2}} \sup_{\partial M_k} \norm{\Re \phi_k}_{h_k} .
    \]
    Therefore it is enough to show that the norm $\norm{\II_k - {}^k \nabla^2 u_k + u_k \, h_k}_{h_k}$ is uniformly bounded by a constant depending only on $k$. The norm of $\II_k$ is equal to $- k^{-1} \sqrt{H_k^2 - 2(k + 1)}$, and $\norm{u_k \, h_k}_{h_k} = \sqrt{2} \, \abs{u_k}$. Therefore we have
    \[
    \norm{\II_k - {}^k \nabla^2 u_k + u_k \, h_k}_{h_k} \leq - k^{-1} \sqrt{\norm{H_k}^2_{\mathscr{C}^0} - 2 (k + 1)} + \norm{{}^k \nabla^2 u_k}_{h_k} + \sqrt{2} \, \norm{u_k}_{\mathscr{C}^0} .
    \]
    Our assertion is now an immediate consequence of Proposition \ref{prop:bounds_mean_curvature} and of Lemma \ref{lem:estimates_uk}.
\end{proof}

\begin{corollary} \label{cor:flow_always_def}
    The flow $\Theta_t$ of the vector field $- \grd_\WP V_k^*$ over $\Teich(\partial M_k)$ is defined for all times $t \in \R$.
\end{corollary}

\begin{proof}
    The assertion follows from the fact that the Teichm\"uller distance is complete, and on the bound shown in Proposition \ref{prop:bound_teichmuller_norm}.
\end{proof}

The last ingredient that we will need for the proof of Theorem \ref{THM:INF_DUAL_VOLUME} is the existence of some lower bound for the dual volume function $V_k^*$. To do so, we will make use of the properties of the dual volume proved in \cite{mazzoli2018the_dual}, and of an upper bound for the length of the bending measure of the boundary of the convex core of a convex co-\hsk compact manifold with incompressible boundary, whose existence has been first proved by Bridgeman \cite{bridgeman1998average}, and it has been improved in later works (see \cite{bridgeman_canary2005bounding}). We will make use of the best result currently known in this direction for convex co-\hsk compact manifolds with incompressible boundary, which is due to Bridgeman, Brock, and Bromberg \cite{bridgeman_brock_bromberg2017}.

\begin{lemma} \label{lem:lower_bound_Vk*}
	For every $k \in (-1,0)$ and for every convex co-\hsk compact hyperbolic $3$-\hsk manifold $M$ with incompressible boundary we have:
	\[
	V_k^*(M) \geq F(k, \chi(\partial M)) ,
	\]
	where $F$ is an explicit function of the curvature $k \in (-1,0)$ and the Euler characteristic of $\partial M$.
\end{lemma}

\begin{proof}
	Since the $k$-\hsk surfaces foliate the complementary of the convex core $\CC M$, a simple application of the geometric maximum principle (see for instance \cite[Lemme~2.5.1]{labourie2000une_lemme}) shows that the $k$-\hsk surface $\partial M_k$ is contained in $\neigh_{\varepsilon_k} \CC M$, the $\varepsilon_k$-\hsk neighborhood of the convex core $\CC M$, for $\varepsilon_k = \arctanh \sqrt{k + 1}$. The dual volume of a convex set is a decreasing function with respect to the inclusion (see \cite[Proposition 2.6]{mazzoli2018the_dual} for a proof of this assertion), therefore the quantity $V_k^*(M)$ is bounded from below by the dual volume of the $\varepsilon_k$-\hsk neighborhood of the convex core. It is not difficult to show that for every $\varepsilon > 0$ we have
	\[
	V^*(\neigh_\varepsilon \CC M) = V(\CC M) - \frac{\length_m(\mu)}{4} (\cosh 2 \varepsilon + 1) - \frac{\pi}{2} \abs{\chi(\partial \CC M)} (\sinh 2 \varepsilon - 2 \varepsilon) ,
	\]
	where $\length_m (\mu)$ denotes the length of the bending measured lamination on the boundary of the convex core of $M$ (see e. g. \cite[Proposition 2.4]{mazzoli2018the_dual}). By \cite[Theorem 2.16]{bridgeman_brock_bromberg2017}, the term $\length_m (\mu)$ is less or equal to $6 \pi \abs{\chi(\partial M)}$. Combining these observations, we deduce that
	\begin{align*}
		V_k^*(M) & \geq V^*(\neigh_{\varepsilon_k} \CC M) \\
		& = V(\CC M) - \frac{\length_m(\mu)}{4} (\cosh 2 \varepsilon_k + 1) - \frac{\pi}{2} \abs{\chi(\partial \CC M)} (\sinh 2 \varepsilon_k - 2 \varepsilon_k) \\
		& \geq - \frac{\length_m(\mu)}{4} (\cosh 2 \varepsilon_k + 1) - \frac{\pi}{2} \abs{\chi(\partial \CC M)} (\sinh 2 \varepsilon_k - 2 \varepsilon_k) \\
		& \geq - \frac{\pi}{2} \abs{\chi(\partial M)} (3 \cosh \varepsilon_k + 3 + \sinh 2 \varepsilon_k - 2 \varepsilon_k) ,
	\end{align*}
	which proves the desired inequality.
\end{proof}

\section{The proof of Theorem \ref{THM:INF_DUAL_VOLUME}} \label{sec:proof_thmA}

This section is dedicated to the proof of the main theorem of our exposition, and to the proof of the optimality of the multiplicative constant appearing in \eqref{eq:lower_bound_volume}.

\begin{proof}[Proof of Theorem \ref{THM:INF_DUAL_VOLUME}]
    Let $M$ be a convex co-\hsk compact hyperbolic $3$-\hsk manifold with incompressible boundary. We denote by $M_t \defin \Theta_t(M)$ the hyperbolic $3$-\hsk manifold obtained by following the flow of the vector field $- \grd_\WP V_k^*$, which is defined for every $t \in \R$ in light of Corollary \ref{cor:flow_always_def}. In order to simplify the notation, we will continue to denote by $V_k^*$ the $k$-\hsk dual volume as a function over the space of quasi-\hsk isometric deformations of $M$. This abuse is justified by the fact that, for every $k \in (-1,0)$, a convex co-\hsk compact manifold is uniquely determined by the hyperbolic structures on its $k$-\hsk surfaces (see Theorem \ref{thm:ksurfaces_homeo}). We have
    \begin{align*}
        V_k^*(M) - V_k^*(M_t) = \int_0^t \norm{\dd V_k^*}^2_{M_s} \dd{s} .
    \end{align*}
    By Lemma \ref{lem:lower_bound_Vk*}, the left hand side of the relation is bounded from above with respect to $t$. In particular, the integral on the right side has to converge as $t$ goes to $+ \infty$. Therefore we can find an unbounded increasing sequence $(t_n)_n$ for which the Weil-\hsk Petersson norm $\norm{\dd{V_k^*}}^2$ evaluated at $M_{t_n}$ goes to $0$ as $n$ goes to $\infty$. Then, by Lemma \ref{lem:lower_bound_gradient}, we have
    \[
    \limsup_{n \to \infty} \int_{\partial M_{t_n, k}} H_k \dd{a_{\I_k}} \leq - 4 \pi k^{-1} \sqrt{k + 1} \abs{\chi(\partial M)} ,
    \]
    where $M_{t_n, k}$ stands for $(M_{t_n})_k$, the region of $M_{t_n}$ enclosed by its $k$-\hsk surfaces. Therefore we deduce:
    \begin{align*}
        V_k^*(M) & \geq \lim_{n \to \infty} V_k^*(M_{t_n}) = \lim_{n \to \infty} \left( V_k(M_{t_n}) - \frac{1}{2} \int_{\partial M_{t_n. k}} H_k \dd{a_{\I_k}} \right) \\
        & \geq \inf_{M' \in \QD(M)} V_k(M') - \frac{1}{2} \limsup_{n \to \infty} \int_{\partial M_{t_n, k}} H_k \dd{a_{\I_k}} \\
        & \geq \inf_{M' \in \QD(M)} V_k(M') + 2 \pi k^{-1} \sqrt{k + 1} \abs{\chi(\partial M)} ,
    \end{align*}
    where $V_k(M')$ denotes the Riemannian volume of the region $M_k'$ of $M'$ enclosed by its $k$-\hsk surface. Observe that the term $2 \pi k^{-1} \sqrt{k + 1} \abs{\chi(\partial M)}$ is equal to $- \frac{1}{2} \int_{\partial M_k'} H_k \dd{a_{\I_k}}$ when the boundary of the convex core of $M'$ is totally geodesic.
    
    Finally, by taking the limit as $k$ goes to $(- 1)^+$, we obtain that $V_\CC^*(M) \geq \inf_{M'} V_\CC(M')$ for every convex co-\hsk compact structure $M$. This proves that
    \[
    \inf_{\QD(M)} V_\CC^* \geq \inf_{\QD(M)} V_\CC .
    \]
    On the other hand, the dual volume $V_\CC^*(M) \defin V_\CC(M) - \frac{1}{2} \length_m(\mu)$ is always smaller or equal to $V_\CC(M)$, so the other inequality between the infima is clearly satisfied.
    
    If $V_\CC^*(M) = V_\CC(M)$, then the length of the bending measured lamination $\mu$ of the convex core of $M$ has to vanish, therefore $\mu = 0$ or, in other words, $\partial \CC M$ is totally geodesic.
\end{proof}

\begin{corollary} \label{cor:optimal_constant}
For every quasi-\hsk Fuchsian manifold $M$ we have $V_\CC(M) \geq \frac{1}{2} \length_m(\mu)$, where $m = m(M)$ and $\mu = \mu(M)$ denote the hyperbolic metric and the bending measure of the boundary of the convex core of $M$, respectively. Moreover, for every positive $\varepsilon$ and for every neighborhood $U$ of a Fuchsian manifold $M_0$ inside $\QD(M_0) = \QD(M)$, there exists a quasi-\hsk Fuchsian manifold $M_\varepsilon$ in $U$ that satisfies $V_\CC(M_\varepsilon) < (\frac{1}{2} + \varepsilon) \length_{m_\varepsilon}(\mu_\varepsilon)$, where $m_\varepsilon = m(M_\varepsilon)$ and $\mu_\varepsilon = \mu(M_\varepsilon)$.
\end{corollary}

\begin{proof}
If $M$ is quasi-\hsk Fuchsian, the infimum of the volume of the convex core over the space of quasi-\hsk isometric deformations $\QD(M)$ is equal to $0$, and it is realized on the Fuchsian locus.

For the second part of the statement, consider $M_0$ a Fuchsian manifold whose convex core is a totally geodesic surface homeomorphic to $\Sigma$ with hyperbolic metric $m_0$. Let $\mappa{\alpha}{[0,1]}{\QD(M)}$ be a path starting at $\alpha(0) = M_0$ and for which the right derivative of the bending measure $\dot{\mu}_0^+$ exists and it is equal to a non-\hsk zero measured lamination on $\Sigma \sqcup \Sigma$. A fairly explicit way to produce such a path is to choose a measured lamination $\lambda \in \MesLam(\Sigma)$ and to consider the deformation of $M_0$ given by the holonomies of pleated surfaces with bending H\"older cocycle equal to $t \lambda$ and hyperbolic metric $m_0$, as $t$ varies in $[0,1]$ (compare with \cite{bonahon1996shearing}). Then, for every $\varepsilon > 0$ we define
\[
f_\varepsilon(t) \defin V_\CC(\alpha(t)) - \left( \frac{1}{2} + \varepsilon \right) \length_{m_t}(\mu_t) = V_\CC^*(\alpha(t)) - \varepsilon \, \length_{m_t}(\mu_t) , \quad t \in [0,1],
\]
where $m_t = m(M_t)$ and $\mu_t = \mu(M_t)$ denote the hyperbolic metric and the bending measure of the boundary of the convex core of $M_t = \alpha(t)$. As shown in \cite[equation (4)]{krasnov2009symplectic}, we have
\[
\left. \dv{t} \length_{m_t}(\mu_t) \right|_{t = 0^+} = \dd{(L_{\mu_0})}(\dot{m}_0) + \ell_{m_0}(\dot{\mu}_0^+) = \ell_{m_0}(\dot{\mu}_0^+) ,
\]
where we are using that $\mu_0 = 0$ (here $\mappa{L_{\mu_0}}{\Teich(\partial \CC M)}{\R}$ is the function that associates with every hyperbolic structure $m$ the length of the $m$-\hsk geodesic realization of $\mu_0$). Then
\begin{align*}
f_\varepsilon(t) & = f_\varepsilon(0) + f_\varepsilon'(0) \, t + o(t;\varepsilon) \\
& = 0 + ( \dd{({V_\CC}^*)_{M_0}}(v) - \varepsilon \, \ell_{m_0}(\dot{\mu}_0^+) ) t + o(t;\varepsilon) \\
& = - \varepsilon \, \ell_{m_0}(\dot{\mu}_0^+) \, t + o(t;\varepsilon) . \tag*{($V_\CC^* \in \mathscr{C}^1$ and $M_0$ minimum)}
\end{align*}
This proves that $f_\varepsilon(t) < 0$ for $t$ sufficiently small (depending on $\varepsilon$), and therefore the existence of a quasi-\hsk Fuchsian manifold $M_\varepsilon$ satisfying the desired properties.
\end{proof}

\emergencystretch=1em

\printbibliography[heading=bibintoc, title={Bibliography}]

\end{document}